\documentclass[a4paper,draft,reqno,12pt]{amsart}
\usepackage[utf8]{inputenc}
\usepackage[T1,T2A]{fontenc}
\usepackage[english]{babel}
\usepackage{amsmath}
\usepackage{amssymb}
\usepackage{amscd}
\usepackage{amsthm}
\usepackage{euscript}
\usepackage[all]{xy}
\newtheorem{theorem}{Theorem}
\newtheorem{proposition}[theorem]{Proposition}
\newtheorem{conjecture}[theorem]{Conjecture}

\newtheorem{lemma}[theorem]{Lemma}

\newtheorem{corollary}[theorem]{Corollary}
\theoremstyle{definition}
\newtheorem{definition}[theorem]{Definition}

\theoremstyle{remark}
\newtheorem {remark}[theorem]{Remark}

\DeclareMathOperator{\lie}{Lie}
\DeclareMathOperator{\adj}{ad}

\DeclareMathOperator{\Aut}{Aut}

\DeclareMathOperator{\SL}{SL}

\DeclareMathOperator{\Der}{Der}

\DeclareMathOperator{\Ker}{Ker}

\def\GG{{\mathbb G}}

\def\CC{{\mathbb C}}
\def\KK{{\mathbb K}}

\def\ZZ{{\mathbb Z}}

\def\OOO{\mathcal{O}}

\renewcommand{\phi}{\varphi}
\renewcommand{\ge}{\geqslant}
\renewcommand{\le}{\leqslant}
\begin{document}
\date{}
\title{Infinite transitivity and polynomial vector fields}
\author{Rafael B. Andrist}
\address{Faculty of Mathematics and Physics, University of Ljubljana, Ljubljana, Slovenia}
\email{rafael-benedikt.andrist@fmf.uni-lj.si}
\author{Ivan Arzhantsev}
\address{Faculty of Computer Science, HSE University, Pokrovsky Boulevard 11, Moscow, 109028 Russia}
\email{arjantsev@hse.ru}

\maketitle
\begin{abstract}
We prove that for many pairs $H_1, H_2$ of root subgroups of the automorphism group $\Aut(\CC^2)$ the diagonal action of the group generated by $H_1, H_2$ on $(\CC^2)^m$ 
has an open orbit for any positive integer $m$. The result is based on the study of the Lie algebra of polynomials in two variables with the standard Poisson bracket. 
\end{abstract} 

\section{Introduction}
\label{sec1}

In this work, we study the properties of multiple transitivity for groups acting on the affine plane $\CC^2$. It follows from the classical Jung theorem that the group of polynomial automorphisms $\Aut(\CC^2)$ is generated by the two-dimensional torus of diagonal matrices and the one-parameter unipotent subgroups of the form
$$
L_r\colon (x+\alpha y^r,y), \quad \alpha\in\CC \quad \text{and}\quad R_s\colon (x,y+\beta x^s), \quad \beta\in\CC,
$$
where $r$ and $s$ run through the non-negative integers. Let us consider the group $G_{r,s}$ generated by the two subgroups $L_r$ and $R_s$. 

We recall that an action of a group $G$ on a set $X$ is infinitely transitive if for any positive integer $m$ and any two $m$-tuples of pairwise distinct points in $X$ there is an element $g\in G$ that sends the first tuple to the second one. It is easy to see that the group $G_{r,s}$ acts on $\CC^2$ with an open orbit. By~\cite{AKZ}, the action of $G_{r,s}$ on its open orbit is infinitely transitive only if $rs=2$. It is asked in~\cite{AKZ} whether $rs=2$ implies that the action is indeed infinitely transitive. The positive answer is given in~\cite{LPS, CT}.  We give another proof of this result in Section~\ref{sec4}.  This proof connects the property of infinite transitivity of an action with the structure of a tangent algebra, and therefore appears promising for further generalizations.

Let us say that an action of a group $G$ on an algebraic variety $X$ is generically infinitely transitive if for any positive integer $m$ the diagonal action of $G$ on $X^m:=X\times\ldots\times X$ ($m$ times) has a Zariski open orbit. Clearly, any infinitely transitive action is generically infinitely transitive. 

The main result of this paper (Theorem~\ref{geninftra}) claims that the action of $G_{r,s}$ on $\CC^2$ is generically infinitely transitive if and only if $rs\ge 2$. It is easy to see
that the condition $rs\le 1$ is equivalent to the fact that $G_{r,s}$ is a linear algebraic group. 

The proof of Theorem~\ref{geninftra} is based on the study of subalgebras of the Lie algebra of polynomials $\CC[x,y]$ with the standard Poisson bracket; see Section~\ref{sec2}. 
We show that for the subalgebra $\lie(x^p,y^q)$ generated by the two monomials $x^p$ and $y^q$ we have 
$$
\lie(x^p, y^q) \cap \CC[x] = x^{p} \CC[x^{pq - p - q}] \quad \text{and} \quad \lie(x^p, y^q) \cap \CC[y] = y^{q} \CC[y^{pq - p - q}],
$$
provided $\min(p,q) \geq 2$ and $\max(p,q) \geq 3$; see Proposition~\ref{prop-gapalgebra}. Moreover, the subalgebra $\lie(x^p,y^q)$ has finite codimension in $\CC[x,y]$ if and only if
$(p,q)=(2,3)$ or $(3,2)$. 

In the same Section~\ref{sec2} we find explicit pairs of derivations that generate the Lie algebra of polynomial vector fields on $\CC^2$ with zero divergence or with constant divergence; see Propositions~\ref{pfirst},~\ref{psecond} and Theorem~\ref{fullalgebra}. These results complement the results of \cite{An-1, An-2, An-3}, where explicit finite generating sets for many Lie algebras of polynomial vector fields are found.

In order to go from Lie algebras to the corresponding transformation groups, we use Proposition~\ref{pimp} on approximation of polynomial flows. This result goes back to Varolin~\cite{Va}, some of its versions were used in \cite{An-1, An-2, An-3}. This approach is applicable not only to the affine plane $\CC^2$, but also to other affine surfaces; see Proposition~\ref{dan} for the case of Danielewski surfaces. 

Denote by $\GG_a$ the additive group of the ground field $\CC$. Any non-trivial regular action $\GG_a\times X\to X$ on an algebraic variety $X$ gives rise to a subgroup $H$ in
$\Aut(X)$ isomorphic to $\GG_a$. Such subgroups are called $\GG_a$-subgroups of $\Aut(X)$.  The subgroups $L_r$ and  $R_s$ are examples of $\GG_a$-subgroups in $\Aut(\CC^2)$.
These are precisely the $\GG_a$-subgroups normalized by the two-dimensional torus of diagonal matrices. Such $\GG_a$-subgroups are called root subgroups. Some results on transitivity properties for groups generated by $\GG_a$-subgroups can be found in~\cite{A, AFKKZ, AKZ} and references therein. 

In Corollary~\ref{two}, we use the fact that the Lie algebra of polynomial vector fields with zero divergence on $\CC^2$ is 2-generated and construct $\GG_a$-subgroups $H_1, H_2$ (one them is not root) such that the group $G$ generated by $H_1, H_2$ is infinitely transitive on~$\CC^2$. 

It is important to notice that the results on Lie algebras in Section~\ref{sec2} are obtained by algebraic methods and are valid over any field of characteristic zero. At the same time, the results on transitivity properties of group actions are based on Propositions~\ref{pimp},~\ref{vpp}, the proofs of which use analytical methods, in particular the Implicit Function Theorem. It would be interesting to obtain criteria for infinite transitivity or generic infinite transitivity in terms of the corresponding Lie algebras applying purely algebraic arguments.
 
\section{Some results on Lie algebras}
\label{sec2}

Let $\KK$ be a field of characteristic zero. Consider the Lie algebra $L$ of all polynomial vector fields $f_1(x,y)\frac{\partial}{\partial x}+f_2(x,y)\frac{\partial}{\partial y}$ on $\KK^2$ with zero divergence, i.e.,
$$
\frac{\partial f_1(x,y)}{\partial x}+\frac{\partial f_2(x,y)}{\partial y}=0. 
$$
\\
 well known that $L$ coincides with the Lie algebra of all volume-preserving vector fields, all symplectic vector fields, and all Hamiltonian vector fields on $\KK^2$. In particular,  vector fields from $L$ are precisely vector fields of the form
$$
V_f:=\frac{\partial f(x,y)}{\partial y}\frac{\partial}{\partial x}-\frac{\partial f(x,y)}{\partial x}\frac{\partial}{\partial y}
$$
for some $f\in\KK[x,y]$. As for the commutator, we have
$$
[V_f, V_h]=V_{\{f,h\}}, \quad \text{where} \quad \{f,h\}=\frac{\partial f(x,y)}{\partial x}\frac{\partial h(x,y)}{\partial y}-\frac{\partial f(x,y)}{\partial y}\frac{\partial h(x,y)}{\partial x}
$$
is the standard Poisson bracket on $\KK[x,y]$. We denote by $A$ the Lie algebra  $(\KK[x,y], \{\cdot, \cdot\})$. 
The map $f\mapsto V_f$ establishes an isomorphism of Lie algebras between $A/I$ and $L$, where $I$ is the central ideal of constants in $A$.

\smallskip

As the first result, we show that the Lie algebra $L$ is 2-generated. We provide a short direct proof. 

\begin{proposition} \label{pfirst}
The Lie algebra $L$ is generated by two locally nilpotent derivations 
$$
\partial_1= y^2\frac{\partial}{\partial x} \quad \text{and} \quad \partial_2= \frac{\partial}{\partial x}-x\frac{\partial}{\partial y}.
$$
\end{proposition} 

\begin{proof}
Notice that $\partial_1= \frac{1}{3}V_{y^3}$ and $\partial_2=\frac{1}{2}V_{x^2+2y}$. So it suffices to show that the polynomials $x^2+2y$ and $y^3$ generate the Lie algebra $A$. 
Let $A_0$ be the Lie subalgebra in $A$ generated by $x^2+2y$ and $y^3$. We have
$$
\{x^2+2y,y^3\}=6xy^2, \ \{x^2+2y, xy^2\}=4x^2y-2y^2, \ \{x^2+2y, 2x^2y-y^2\}=4x^3-12xy,  
$$
$$
\{x^2+2y, x^3-4xy\}=6y-12x^2, \ \ \{x^2+2y,y-2x^2\}=10x, \ \ \{x^2+2y, x\}=-2.
$$
\smallskip
Since the polynomials $x^2+2y$ and $y-2x^2$ are not proportional, we conclude 
that $1,x,x^2,y$ lie in $A_0$. Further, we obtain
$$
\{\{x^2,y^l\},y^l\}=2l^2 y^{2(l-1)}.  
$$
This implies that if $y^l\in A_0$ then $y^{2(l-1)}\in A_0$ and $2(l-1)>l$ for $l\ge 3$. Since we have
$$
\{x,y^l\}=ly^{l-1},
$$
we conclude that $y^s\in A_0$ for all $s\in\ZZ_{\ge 0}$. In the same way we have 
$$
\{\{\{y^3,x^k\},x^k\},x^k\} =6 k^3 x^{3(k-1)}.  
$$
This shows that if $x^k\in A_0$ then $y^{3(k-1)}\in A_0$ and $3(k-1)>k$ for $k\ge 2$. Since we have
$$
\{x^k,y\}=kx^{k-1},
$$
we conclude that $x^m\in A_0$ for all $m\in\ZZ_{\ge 0}$. Finally, the equality
$$
\{x^m,y^s\}=msx^{m-1}y^{s-1}, 
$$
implies that all monomials in $x,y$ are in $A_0$ and so $A_0$ coincides with $A$. 
\end{proof} 

\begin{remark}
The fact that the Lie algebra $L$ is 2-generated can also be deduced from~\cite[Theorem~28]{AH} with $n=1$. 
\end{remark} 

Let us take a closer look at the structure of the Lie algebra $A=(\KK[x,y], \{\cdot, \cdot\})$. 

\begin{lemma} \label{lem1}
We have
$$
\{x^p y^q, x^r y^s\} = (p s - q r)  x^{p+r-1} y^{q+s-1} = 
\begin{vmatrix}
p & q \\ 
r & s 
\end{vmatrix} 
x^{p+r-1} y^{q+s-1} . 
$$
\end{lemma}

The proof is straightforward. 

\smallskip

Let us define an operator $\adj_f\colon A\to A$ with $f\in A$ by $\adj_f(g)=\{f,g\}$. 

\begin{lemma} \label{lem2}
Let $p,q,n$ be positive integers with $n\le pq$ and $\epsilon$ be an element of $\KK$. Then 
\[
\adj_{x^p - \varepsilon y}^{n}(y^q) = \sum_{k=0}^s c_{k,n} x^{p(q-k) - n} y^k,
\]
where $s=\lfloor q-\frac{n}{p}\rfloor$. Moreover, we have 
the recursion relation
\[
c_{k, n+1} = \varepsilon (p (q - k) - n) c_{k,n} + p (k+1) c_{k+1,n}
\]
with $c_{k,n}=0$ for $k>s$.
\end{lemma}

\begin{proof} 
We use the induction on $n$. With $n=0$ we obtain
$c_{q,0} = 1$ and $c_{k,0} = 0$ for $k \neq q$. 
For the induction step we have 
$$
\{ x^p - \varepsilon y, \sum_{k=0}^s c_{k,n} x^{p(q-k) - n} y^k \} = \sum_{k=0}^s c_{k,n} \{ x^p, x^{p(q-k) - n} y^k \} - \sum_{k=0}^s \varepsilon c_{k,n}\{ y, x^{p(q-k) - n} y^k \} =
$$
$$
= \sum_{k=0}^s  k p c_{k,n} x^{p(q-k) - n + p - 1} y^{k-1} + \sum_{k=0}^s \varepsilon (p(q-k) - n) c_{k,n}  x^{p(q-k) - n - 1} y^k =
$$
$$
= \sum_{k=0}^{s'} c_{k,n+1} x^{p(q-k) - (n+1)} y^k, 
$$
where $s'=\lfloor q-\frac{n+1}{p}\rfloor$. This leads to the required relation. 
\end{proof}

If $\varepsilon\ne 0$, we apply the automorphism $(x,y)\mapsto (x,\varepsilon y)$ of the algebra $\KK[x,y]$ and assume that $\varepsilon=1$. 

\begin{lemma} \label{lem3}
We have 
$$
\adj_{x^p - y}^{p(q-1)}(y^q) = c_{0,p(q-1)} x^{p} + c_{1,p(q-1)}	 y \quad \text{and} \quad \adj_{x^p - y}^{pq}(y^q) = c_{0,pq}. 
$$
Moreover, $c_{0,pq} \neq 0$ and $c_{0,p(q-1)} +c_{1,p(q-1)}\ne 0$. \\
\end{lemma}

\begin{proof}
The first two equalities follow directly from Lemma~\ref{lem2}. Since $c_{q,0} = 1$ and $c_{k,0} = 0$ for $k \neq q$, we conclude from the recursion formula that
$c_{0,i}=0$ for $i<q$ and $c_{0,q}\ne 0$.  We have $c_{0, n+1} = (pq - n) c_{0,n} + p c_{1,n}$, so $c_{0,i}\ne 0$ for $q\le i\le pq$ as a sum of two non-negative summands
such that at least one of them is positive. So we obtain $c_{0,pq}\ne 0$. 

If $c_{0,p(q-1)} +c_{1,p(q-1)}=0$ then the element $\adj_{x^p - y}^{p(q-1)}(y^q)$ is proportional to $x^p-y$.  This implies 
$$
0=\adj_{x^p - y}^{p(q-1)+1}(y^q)=\ldots=\adj_{x^p - y}^{pq}(y^q)=c_{0,pq},
$$
a contradiction. 
\end{proof}

\begin{lemma} \label{lem4} 
We have
$$
\adj_{x^p}^{q}(y^q) = p^q q! x^{pq - q}. 
$$
\end{lemma}

\begin{proof}
With $\varepsilon = 0$ the recursion relation in Lemma~\ref{lem2} has the form $c_{k, n+1} = p (k+1) c_{k+1,n}$. Since $c_{q,0} = 1$ and $c_{k,0} = 0$ for $k \neq q$,
we obtain $c_{0,q}=  p^q q!$ and $c_{i,q}=0$ for $i\ne 0$. 
\end{proof}

Let us denote by $\lie(f,g)$ the subalgebra of the Lie algebra $A$ generated by $f,g\in A$. 

\begin{proposition}
\label{prop-gapalgebra}
If $\min(p,q) \geq 2$ and $\max(p,q) \geq 3$ then 
$$
\lie(x^p, y^q) \cap \KK[x] = x^{p} \KK[x^{pq - p - q}]
$$
and
$$
\lie(x^p, y^q) \cap \KK[y] = y^{q} \KK[y^{pq - p - q}].
$$
\end{proposition}

\begin{proof}
Let $S = \lie(x^p, y^q)$. 
We first prove the sufficiency of the conditions. 
\smallskip
\begin{enumerate}
\item By Lemma~\ref{lem4},  $S \ni p^q q! x^{pq-q} = \adj_{x^p}^{q}(y^q)$.
\smallskip
\item We show that $x^{k(pq-q)-(k-1)p} \in S \Longrightarrow x^{(k+1)(pq-q)-kp} \in S$ for $k \in \ZZ_{>0}$:
\smallskip
$$
\adj_{x^p}^{q-1}\adj_{y^q} (x^{k(pq-q)-(k-1)p}) = \adj_{x^p}^{q-1}(-q(k(pq-q)-(k-1)p) x^{k(pq-q)-(k-1)p-1} y^{q-1}) =
$$
$$
= -(k(pq-p-q)+p) p^{q-1} q! x^{(k+1)(pq-q)-kp}. 
$$
\end{enumerate}
Since $\min(p,q) \geq 2$, $\max(p,q) \geq 3$, all the coefficients and exponents are positive.
We conclude that $S$ contains the monomials $x^{k(pq-p-q) + p}$ for any non-negative integer $k$. 
We can also interchange the roles of $x^p$ and $y^q$.

For the necessity of the conditions, observe that all possible Lie combinations are obtained through a combination of the adjoint actions $\adj_{x^p}$ and $\adj_{y^q}$ which map monomials to monomials. Applying the operators $\adj_{x^p}$ and $\adj_{y^q}$ to $x^p$ and $y^q$, we obtains monomials with exponents  
$$
(p,0) + k (p-1,-1) + l (-1,q-1) \,:\, k, l \in \ZZ_{\ge 0} 
$$
and
$$
(0,q) +  k (p-1,-1) + l (-1,q-1) \,:\, k, l \in \ZZ_{\ge 0}, 
$$
respectively.  For $(m,0)$, this yields either $m = p + l (pq-p-q)$ for the first variant or $m = p + (l + 1)(pq-q-p)$ for the second variant. 
Similarly, for $(0,n)$ we have ${n = q + (k+1)(pq-p-q)}$ or $n = q + k (pq-p-q)$. 
\end{proof}

\begin{corollary}
If $p = 3$ and $q = 2$ then
$$
\KK[x,y] = \langle 1, x, y, xy, x^2 \rangle\oplus \lie(x^3, y^2). 
$$
\end{corollary}

\begin{proof}
We have $pq - p - q = 1$ and Proposition~\ref{prop-gapalgebra} implies $x^3 \KK[x] + y^2 \KK[y] \subseteq\lie(x^3, y^2)$. 
The inclusion of mixed powers $x^ay^b$ follows from Lemma~\ref{lem1}.
\end{proof}

Since $pq-p-q=1$ is equivalent to $(p-1)(q-1)=2$ and
$$
\lie(x,y^q)=\langle 1,x,y,y^2,\ldots,y^q\rangle, \quad \lie(x^2, y^2) = \langle xy, x^2,y^2 \rangle, 
$$
we obtain

\begin{corollary}
The subalgebra $\lie(x^p,y^q)$ has finite codimension in $\KK[x,y]$ if and only if $(p,q) = (2,3)$ or $(3,2)$.  
\end{corollary}

The following result is a generalization of Proposition~\ref{pfirst}. 

\begin{theorem} \label{fullalgebra}
If $\min(p,q) \geq 2$ and $\max(p,q) \geq 3$ then
$$
\lie(x^p - y, y^q) = \KK[x,y].
$$
\end{theorem}

\begin{proof}
Since we have
\[
\adj_{x^p - y}^{p(q-1)}(y^q) = c_{0,p(q-1)} x^{p} + c_{1,p(q-1)} y
\]
and $c_{0,p(q-1)} + c_{1,p(q-1)}\ne 0$ (Lemma~\ref{lem3}), we obtain $x^p, y \in \lie(x^p - y, y^q)$. 
Lemma~\ref{lem1} implies
$$
\adj_{y}^n(x^p) = (-1)^n p \cdots (p-n+1) x^{p-n} \in \lie(x^p - y, y^q) 
$$
and
$$
\adj_{x}^n(y^q) = q \cdots (q-n+1) y^{q-n} \in \lie(x^p - y, y^q). 
$$
We conclude that the elements $x^p, x^{p-1},\ldots, x$ and $y^q, y^{q-1},\ldots, y$ are contained in ${\lie(x^p -  y, y^q)}$. 
Since $\min(p,q) \geq 2$ and $\max(p,q) \geq 3$, either the elements $x^2+2y, y^3$ or the elements $y^2+2x, x^3$ are contained in  
$\lie(x^p - y, y^q)$.  It is shown in the proof of Proposition~\ref{pfirst} that every such pair generates the algebra $A$. 
\end{proof}

Finally, let us give one more result on Lie algebras. Let $\Delta=x\frac{\partial}{\partial x} + y\frac{\partial}{\partial y}$. Consider the Lie algebra $\widehat{L}$ 
of polynomial vector fields on $\KK^2$ with constant divergence.  By \cite[Proposition~15.7.2]{FK}, the Lie algebra $\widehat{L}$  coincides with the tangent
algebra of the group of polynomial automorphisms $\Aut(\KK^2)$. Clearly, we have
$$
\widehat{L} = \langle \Delta \rangle \oplus L. 
$$
Let us show that the Lie algebra $\widehat{L}$ is 2-generated, which is a natural extension of Proposition~\ref{pfirst}. 

\begin{proposition} \label{psecond}
The Lie algebra $\widehat{L}$ is generated by two derivations 
$$
\widehat{\partial_1}= \Delta + 3y^2\frac{\partial}{\partial x} \quad \text{and} \quad \partial_2= \frac{\partial}{\partial x}-x\frac{\partial}{\partial y}.
$$
\end{proposition}

\begin{proof}
Note that $[\Delta, V_{x^ay^b}]=(a+b-2)V_{x^ay^b}$.  It is easy to see that the Lie algebra $A=(\KK[x,y], \{\cdot, \cdot\})$ is
$\ZZ$-graded:
$$
A=\oplus_{i\in\ZZ} A_i, \quad \text{where} \quad  A_i=\langle x^ay^b \, ;  \, a+b-2=i\rangle. 
$$
Define the linear operator $\delta$ on $A$ given by $\delta(x^ay^b)=(a+b-2)x^ay^b$. We may consider the semidirect product
of Lie algebras:
$$
\widehat{A}:=\langle \delta\rangle \rightthreetimes A, \quad [\delta, x^ay^b]=(a+b-2)x^ay^b. 
$$
Then we have $[\Delta, V_f]=V_{\{\delta,f\}}$ for any polynomial $f\in\KK[x,y]$. Again the projection 
$$
\widehat{A}\to\widehat{L}, \quad \alpha\delta+f\mapsto \alpha\Delta+V_f, \quad \alpha\in\KK
$$
establishes an isomorphism $\widehat{A}/\widehat{I}\cong\widehat{L}$, where $\widehat{I}=0+\KK$ is the ideal of constants. So it suffices to prove that the elements $x^2+2y$ and $\delta+y^3$ generate the algebra
$\widehat{A}$. We have 
$$
\{x^2+2y,\delta+y^3\}=2y+6xy^2, \ \{x^2+2y, y+3xy^2\}=2x+12x^2y-6y^2, 
$$
$$
\{x^2+2y, x+ 6x^2y - 3y^2\}=-2+12x^3-36xy, \  \{x^2+2y, x^3-3xy\}=6y-12x^2, 
$$
$$
\ \ \{x^2+2y, y-2x^2\}=10x, \ \ \{x^2+2y, x\}=-2.
$$
Since the polynomials $x^2+2y$ and $y-2x^2$ are not proportional, we conclude 
that $1,x,x^2,y$ lie in the Lie algebra $\widehat{A}_0$ generated by $x^2+2y$ and $\delta+y^3$. Further, we have
$$
\{\{x^2,\delta+y^3\},y\}=\{6xy^2,y\}=6y^2 \quad \text{and} \quad \{xy^2,y^2\}=2y^3.
$$
We conclude that the elements $\delta, x^2+2y, y^3$ are in $\widehat{A}_0$, and the claim follows from Proposition~\ref{pfirst}. 
\end{proof}

\begin{remark}
While the vector field $\widehat{\partial_1}=(x+3y^2)\frac{\partial}{\partial x}+y\frac{\partial}{\partial y}$  is not the velocity field of any one-parameter polynomial subgroup of automorphisms of $\KK^2$, over the field of complex numbers it is the velocity field of one-parameter subgroup of holomorphic automorphisms of $\CC^2$,
namely,
$$
(x,y)\mapsto ((x-3y^2) e^t+3y^2 e^{2t}, y e^{t}). 
$$
\end{remark}

\begin{remark}
Developing~\cite[Theorem~11]{An-1}, it is shown in~\cite{Be} that the Lie algebra of all polynomial vector fields on $\KK^n$ is 2-generated. Three generators of the Lie algebra of all volume-preserving polynomial vector fields or, equivalently, of the Lie algebra of all polynomial vector fields with divergence zero on $\KK^n$ with $n\ge 2$ are given in~\cite[Theorem~14]{An-1}. We do not know whether the Lie algebra of all polynomial vector fields with zero divergence and the Lie algebra of all polynomial vector fields with constant divergence on $\KK^n$ are 2-generated for $n\ge 3$. 
\end{remark} 

\section{Infinite transitivity and root subgroups}
\label{sec3}

In this section we assume that the ground field $\KK$ is an algebraically closed field of characteristic zero. Let $X$ be an affine algebraic variety of dimension at least~2 and $T$
be an algebraic torus with the character lattice $M$. Assume that we have an effective action $T\times X\to X$. Let $H_1,\ldots, H_k$ be $\GG_a$-subgroups in $\Aut(X)$ normalized by $T$.
Such subgroups are called root subgroups. Let $e_1,\ldots,e_k\in M$ be characters of $T$ which correspond to the action of $T$ on $H_1,\ldots,H_k$ by conjugation. This means that
$$
tH_i(a)t^{-1}=H_i(e_i(t)a) \quad \forall t\in T, a\in\KK. 
$$
Consider the subgroup $G=\langle H_1,\ldots,H_k\rangle$ in $\Aut(X)$.  Denote by $X^m$ the direct product $X\times\ldots\times X$ ($m$ times). 

\begin{definition}
An action of a group $G$ on an algebraic variety $X$ is said to be \emph{generically infinitely transitive} if for any positive integer $m$ the diagonal action of $G$ on $X^m$ has an open orbit. 
\end{definition}

Assume that the group $G$ has an open orbit $\OOO$ on $X$. We are interested when $G$ acts on $\OOO$ infinitely transitively or, weaker, when the action of $G$ on $X$
is generically infinitely transitive. 

\smallskip

The following result is a conceptual generalization of~\cite[Remark~5.7]{AKZ}. Also it proves one implication in~\cite[Conjecture~2]{CT}. 

\begin{proposition} \label{prop-a}
Assume that the group $G$ acts on $X$ with an open orbit $\OOO$ and the $G$-action on $\OOO$ is 2-transitive. Then the characters $e_1,\ldots,e_k$ generate the lattice $M$.
\end{proposition}

\begin{proof}
Assume that $e_1,\ldots,e_k$ generate a proper sublattice in $M$. Then there is a non-unit element $t_0\in T$ with 
$$
e_1(t_0)=\ldots=e_k(t_0)=1. 
$$
The element $t_0$ commutes with all elements in each $H_i$, and so with all elements in $G$. Thus $t_0$ permutes $G$-orbits on $X$. In particular, the automorphism $t_0$ preserves the open orbit $\OOO$. Now the assertion follows from the next set-theoretic lemma.
\end{proof}

\begin{lemma}
Assume that a group $G$ acts on a set $\OOO$ transitively and $|\OOO|\ge 3$. If there is a non-identical bijection $t_0\colon \OOO\to\OOO$, which commutes with any element in $G$, then the $G$-action on $\OOO$ is not 2-transitive. 
\end{lemma}

\begin{proof}
Take a point $x\in\OOO$ such that $t_0x\ne x$ and consider the pair $(x,t_0x)$. We have 
$$
(gx,gt_0x)=(gx,t_0gx) \quad \forall g\in G.
$$
This shows that an element of $G$ can not send the pair $(x,t_0x)$ to a pair $(x,y)$ with ${x\ne y\ne t_0x}$. 
\end{proof}

Now we concentrate on the following particular situation. Fix non-negative integers $r,s$ and consider two one-parameter subgroups of automorphisms of the affine plane $\KK^2$:
$$
L_r\colon (x+\alpha y^r,y)\quad \alpha\in\KK \quad \text{and}\quad R_s\colon (x,y+\beta x^s), \quad \beta\in\KK.
$$
Denote by $G_{r,s}$ the subgroup of the group $\text{Aut}(\KK^2)$ generated by $L_r$ and $R_s$. 

Notice that the torus $T=(\KK^{\times})^2$ acts on $\KK^2$ as $(t_1,t_2)\circ (x,y)=(t_1x,t_2y)$. 
It follows from Proposition~\ref{prop-a} that the subgroup $G_{r,s}$ can act infinitely transitively on its open orbit only if the vectors $e_1=(-1,r)$ and $e_2=(s,-1)$ generate the lattice $M=\ZZ^2$ or, equivalently, $rs=0$ or $rs=2$. In the first case the group $G_{r,s}$ acts along one coordinate by parallel translations, so the action preserves the difference of the corresponding coordinates, and the action is not 2-transitive~\footnote{This gives a counterexample to Conjecture~1 in~\cite{CT}.}. 

It is asked in~\cite[Section~5.1]{AKZ} whether $rs=2$ implies that the action of $G_{r,s}$ on $\KK^2\setminus\{0\}$ is infinitely transitive. It is proved in~\cite[Corollary~21]{LPS} and in a more general form in~\cite[Theorem~1]{CT} that this action is indeed infinitely transitive. So we have the following result.

\begin{theorem} \label{teor2}
The group $G_{r,s}$ acts on an open subset $\OOO$ in $\KK^2$ infinitely transitively if and only if $rs=2$. In the later case, we have $\OOO=\KK^2\setminus\{0\}$. 
\end{theorem}

In the next section we give one more proof of this result over the field of complex numbers and also obtain the following theorem. 

\begin{theorem} \label{geninftra}
The action of the group $G_{r,s}$ on $\KK^2$ is generically infinitely transitive if and only if $rs\ge 2$.
\end{theorem}

Let us finish this section with the following conjecture.  For an affine variety $X$, we denote by $N(X)$ the subalgebra of the Lie algebra of derivations $\Der(\KK[X])$ generated by all locally nilpotent derivations. Assume that $X$ has dimension at least $2$ and $D_1,\ldots D_k$ are locally nilpotent derivations on $\KK[X]$. Denote by $H_i$ the one-parameter subgroup of $\Aut(X)$ of the form $\exp(tD_i)$, $t\in\KK$. Also denote by $\lie(D_1,\ldots D_k)$ the Lie algebra generated by $D_1,\ldots D_k$. Assume that the group $G:=\langle H_1,\ldots, H_k\rangle$ acts on $X$ with an open orbit $\OOO$

\begin{conjecture}
If $\lie(D_1,\ldots D_k)$ has finite codimension in $N(X)$, then the group $G$ acts on $\OOO$ infinitely transitively. 
\end{conjecture} 

\section{Infinite transitivity on affine surfaces}
\label{sec4}

In this section we assume that the ground field is the field $\CC$ of complex numbers.  Our main aim is to prove Theorem~\ref{geninftra}.
We begin with the following elementary lemma.
\begin{lemma}
\label{lem-interpolate}
Let $z_1, \dots, z_{m-1}, z_m \in \CC$ and $d_0, d \in \ZZ_{\ge 0}$. If $z_m \neq 0$ and 
$$
z^d_m \neq z^d_j, \quad j = 1,\dots, m-1,
$$
then there exists a polynomial $f \in z^{d_0} \CC[z^d]$ with $f(z_1) = \dots = f(z_{m-1}) = 0$ and $f(z_{m}) = 1$.
\end{lemma}
\begin{proof}
We choose 
$$
f(z) = c_0 z^{d_0} \prod_{j=1}^{m-1} (z^d - {z_j}^d).
$$
We can find $c_0 \in \CC$ such that $f(z_{m}) = 1$ under the assumptions made on $z_{m}$.
\end{proof}

Now let $X$ be a connected smooth affine algebraic surface. In this section we call the subgroup $\exp(tD)$, $t\in\CC$, associated with a locally nilpotent derivation $D$ on $\CC[X]$, the \emph{flow}
 of $D$.
 
\begin{lemma} \label{pol}
Assume that there exist two locally nilpotent derivations $\Theta_1$ and $\Theta_2$ on $\CC[X]$ 
that span the tangent space $T_xX$ in every point $x$ of an open subset $\Omega\subseteq X$.
Then there are functions $x_1,x_2\in\CC[X]$ such that $\Ker(\Theta_1)=\CC[x_1]$ and 
$\Ker(\Theta_2)=\CC[x_2]$. 
\end{lemma}

\begin{proof}
Since the variety $X$ is smooth, the algebra $\CC[X]$ is integrally closed. The subalgebra $\Ker(\Theta_1)$ coincides with the algebra of invariants of a $\GG_a$-action, so it is 
integrally closed as well. By~\cite[Theorem~1]{Win}, $\Ker(\Theta_1)$ is the algebra of regular functions on a quasi-affine variety $C$. But $C$ has dimension $1$~\cite[Principle~11(e)]{Fr}, so it is a smooth affine curve. Take any point $x_0\in\Omega$. The map
$$
\CC^2\to X, \quad (y_1,y_2)\mapsto \exp(y_1\Theta_1)\exp(y_2\Theta_2) x_0
$$
is a dominant morphism, so the surface $X$ is rational. Since the inclusion ${\Ker(\Theta_1)\subseteq\CC[X]}$ induces a dominant morphism $X\to C$, the curve $C$ is rational as well. 

Any invertible regular function $f$ on $X$ is invariant with respect to the flows of $\Theta_1$ and $\Theta_2$, see~\cite[Principle~1(b)]{Fr}. So the flows preserve the curves where $f$ is constant. This contradicts to the fact that
$\Theta_1$ and $\Theta_2$ generate the tangent space at a generic point. We conclude that $\CC[X]$ contains no non-constant regular function. The same holds for $\CC[C]$. So $C$ is an affine line and $\Ker(\Theta_1)$ is the algebra of polynomials in one variable. The same arguments work for $\Ker(\Theta_2)$.
\end{proof} 

We use notation of Lemma~\ref{pol} and for any positive integers $m, d$ and an open subset $\Omega\subseteq X$ we define a subset in $X^m$ as
$$
\Omega^m_d=\{(P_1,\ldots,P_m)\in \Omega^m \, | \, x_i^d(P_s)\ne x_i^d(P_l)\ne 0,\, i=1,2,\, s\ne l,\,  s,l=1,\ldots, m\}.
$$

\begin{lemma}
\label{lem-vf-spanning}
Under the assumptions of Lemma~\ref{pol}, let $(P_1, \dots, P_m)\in\Omega^m_d$.  Then there exist locally nilpotent derivations $V_{j,k}$ on $\CC[X]$ with $j =1,\ldots, m$ and $i =1, 2$
of the form
$$
V_{j,i} = f_{j,k} \Theta_i, \quad f_{j,i} \in x_i^{d_i} \CC[x_i^d],
$$
such that for every $j =1,\ldots, m$ we have  
$$
\mathop{\mathrm{span}} \{ (V_{j,i})_{P_j} \}_{i=1,2} = T_{P_j} X
$$
and for every $l \neq j$ and every $i$ we have
$$
(V_{j,i})_{P_l} = 0.
$$
\end{lemma}
\begin{proof}
By Lemma \ref{lem-interpolate} there exist functions $f_{j,i} \in x_i^{d_i} \CC[x_i^d]$ such that
$$
f_{j,i}(x_k(P_l)) = \delta_{j l}.
$$
Then the vector fields $V_{j,i} = f_{j,i} \Theta_i$ have the desired properties. 	
\end{proof}

\begin{proposition} \label{pimp}
Let $X$ be a connected smooth affine algebraic surface. Assume that there exist two locally nilpotent derivations $\Theta_1$ and $\Theta_2$ on $\CC[X]$ with $\Ker(\Theta_i)=\CC[x_i]$, 
$i=1,2$, that span the tangent space $T_xX$ in every point $x$ of an open subset $\Omega\subseteq X$. Assume also that there exist locally nilpotent derivations $D_1,\ldots,D_k$ on
$\CC[X]$ such that 
$$
\lie(D_1,\ldots,D_k)\supseteq x_1^{d_1}\CC[x_1^d]\Theta_1\oplus x_2^{d_2}\CC[x_2^d]\Theta_2
$$
for some non-negative integers $d_1,d_2$ and some positive integer $d$. 
Then for every positive integer $m$ the group $G$ generated by the flows of  $D_1,\ldots,D_k$ acts on $X^m$ with an open orbit $\OOO_m$, and the action of $G$ on $\OOO_m$ is infinitely transitive. Moreover, the orbit $\OOO_m$  contains the subset $\Omega^m_d$. 
\end{proposition}

\begin{proof}
Throughout this proof, all spaces are endowed with the Euclidean topology or the subspace topology inherited thereof.

Let $P_1, \dots, P_m, P_{m+1} \in \Omega^{m+1}_d$ and $a_j = x_1(P_j)$, $b_j = x_2(P_j)$. It suffices to construct an automorphism $\Phi \colon X \to X$ such that $\Phi(P_j) = P_j$ for every 
$j =1, \dots, m-1$ and $\Phi(P_m) = P_{m+1}$. Indeed, the subset $\Omega^m_d$ is open in $X^m$, so the group $G$ acts with an open orbit which is the union of translations of $\Omega^m_d$. 

We choose a path $\gamma \colon [0,1] \to X$ from $P_m$ to $P_{m+1}$ such that its image satisfies
\begin{enumerate}
\item $x_1|_{\gamma((0,1))} \colon \gamma((0,1))\to \CC \setminus \{\zeta a_1, \dots \zeta a_{m+1} \,:\, \zeta^d = 1\}$;
\item $x_2|_{\gamma((0,1))} \colon \gamma((0,1))\to \CC \setminus \{\zeta b_1, \dots \zeta b_{m+1} \,:\, \zeta^d = 1\}$.
\end{enumerate}
Such a path $\gamma$ exists since
\[X \setminus \left( x_1^{-1}(\{\zeta a_1, \dots \zeta a_{m+1} \,:\, \zeta^d = 1\}) \cup x_2^{-1}(\{\zeta b_1, \dots \zeta b_{m+1} \,:\, \zeta^d = 1\}) \right)\]
is path-connected. 

For each $\tau \in [0,1]$ we apply Lemma \ref{lem-vf-spanning} to the points $P_1, \dots, P_{m-1}, \gamma(\tau)$ and obtain the vector fields $V_{j,i,\tau}$. Denote their flows by $\varphi_{j,i,\tau}^t$, where $t$ stands for the flow time. The map $\CC^m\otimes \CC^2 \to X^m$ defined by
\[
(t_{j,i})_{j=1, \dots, m; i=1,2} \mapsto \varphi_{m,2,\tau}^{t_{m,2}} \circ \varphi_{m,1,\tau}^{t_{m,1}} \circ \dots \circ \varphi_{1,2,\tau}^{t_{1,2}} \circ \varphi_{1,1,\tau}^{t_{1,1}}(P_1, \dots P_{m-1}, \gamma(\tau))
\]
it is a submersion in a neighborhood of $0$. Indeed, its Jacobian determinant is non-zero. 

By assumption, each of the vector fields $V_{j,i,\tau}$ is a finite Lie combination of the locally nilpotent derivations $D_1,\ldots, D_k$. By \cite[Corollary~4.8.4]{For}, a finite composition $\Phi_{j,i,\tau}^t$ of the flows of 
$D_1, \ldots, D_k$ approximates the flow $\varphi_{j,i,\tau}^t$ of $V_{j,i,\tau}$ on a relatively compact neighborhood of $P_1, \dots, P_{m-1}, \gamma(\tau)$ and for a small enough flow-time $|t| < \delta$ arbitrarily well. Hence, we can also approximate the above composition 
$$
\varphi_{m,2,\tau}^{t_{m,2}} \circ \varphi_{m,1,\tau}^{t_{m,1}} \circ \dots \circ \varphi_{1,2,\tau}^{t_{1,2}} \circ \varphi_{1,1,\tau}^{t_{1,1}}(P_1, \dots P_{m-1}, \gamma(\tau))
$$ 
by a finite composition 
$$
\Phi_{m,2,\tau}^{t_{m,2}} \circ \Phi_{m,1,\tau}^{t_{m,1}} \circ \dots \circ \Phi_{1,2,\tau}^{t_{1,2}} \circ \Phi_{1,1,\tau}^{t_{1,1}}(P_1, \dots P_{m-1}, \gamma(\tau))
$$ 
of flows of the locally nilpotent derivations $D_1, \ldots, D_k$. We denote this composition by $\Phi_{\tau} \colon \CC^{m} \otimes \CC^{2} \to X^m$, considered as a function depending on the flow times $t_{j,k}$. This approximation will eventually become a submersion as well on a neighborhood $U_\tau$ of ${0 \in \CC^{m} \otimes \CC^{2}}$. Indeed, if a holomorphic map approximates another holomorphic map uniformly on compacts, then so do their derivatives; In particular, a good enough approximation will need to have non-zero Jacobian determinant. Since $[0,1]$ is compact, we can cover $\gamma([0,1])$ by finitely many sets of the form $\Phi_{\tau_1}(U_{\tau_1}), \dots, \Phi_{\tau_s}(U_{\tau_s})$. By the Implicit Function Theorem, we now find for each of path-times $\tau_1, \dots, \tau_s$ the parameters in $\CC^m \otimes \CC^2$ such that we can successfully move $P_m$ to $P_{m+1}$ along the path $\gamma$ in finitely many steps. The resulting composition is the desired automorphism which is algebraic as a finite composition of flows of locally nilpotent derivations. 
\end{proof} 

\begin{remark}
In~\cite[Conjecture~5.23]{AKZ}, the following result is expected. Let $X$ be an affine variety, $D_1,\ldots,D_k$ be locally nilpotent derivations on $\KK[X]$, and $H_1,\ldots,H_k$ be their flows. Consider one more locally nilpotent derivation $D$ on $\KK[X]$ with flow $H$. Then the subgroup $H$ lies in the closure of the subgroup $G=\langle H_1,\ldots,H_k\rangle$  with respect to the ind-topology on $\Aut(X)$ if and only if $D\in\lie(D_1,\ldots,D_k)$. The reasoning used in the proof of Proposition~\ref{pimp} may be considered as analytical arguments in favor of this conjecture.
\end{remark}

\smallskip

Now we are ready to prove the main results. 

\begin{proof}[Proof of Theorem~\ref{geninftra}]
If $rs\le 1$ then $G_{r,s}$ is a linear algebraic group, and this group can not act with an open orbit on varieties $(\CC^2)^m$ of arbitrarily high dimension. Let us assume further that $rs\ge 2$. 

In the notation of Proposition~\ref{pimp} take $X=\CC^2$ with $\Theta_1=\frac{\partial}{\partial x}$ and $\Theta_2=\frac{\partial}{\partial y}$. Then $\Omega=\CC^2$ and $x_1=y, x_2=x$. 
Also we let $k=2$ and $D_1=y^r\frac{\partial}{\partial x}$, $D_2=x^s\frac{\partial}{\partial y}$. Under the identification of the Lie algebra $L$ with $A/I$ in Section~\ref{sec2},
we identify $D_1$ with $\frac{1}{r+1}y^{r+1}$ and $D_2$ with $\frac{1}{s+1}y^{s+1}$.  By Proposition~\ref{prop-gapalgebra}, we obtain
$$
\lie(D_1, D_2)\supseteq x_1^{d_1}\CC[x_1^d]\Theta_1\oplus x_2^{d_2}\CC[x_2^d]\Theta_2
$$
with 
$$
d=(r+1)(s+1)-(r+1)-(s+1)=rs-1, \quad d_1=r+1, \quad d_2=s+1.
$$  
Now Proposition~\ref{pimp} claims that the action of the group $G_{r,s}$ on $\CC^2$ is generically infinitely transitive provided $rs\ge 2$. 
\end{proof}

\smallskip

\begin{proof}[Proof of Theorem~\ref{teor2}]
By Theorem~~\ref{geninftra} and Proposition~\ref{pimp}, the group $G_{r,s}$ acts on $(\CC^2)^m$ with an open orbit for all $m$, and the open orbit contains the subset $\Omega^m_d$. 
With $rs=2$ we have $d=1$, so the subset $\Omega^m_1$ consists of collections
$$
((x_1,y_1), \ldots, (x_m,y_m)) \quad \text{with} \ x_i\ne 0, \ y_i\ne 0, \ x_i\ne x_j, \ y_i\ne y_j \ \text{for\, all} \ 1\le i\ne j\le m. 
$$
So it remains to show that any point $((x_1,y_1), \ldots, (x_m,y_m))\in(\CC^2\setminus\{0\})^m$ can be sent to a point in $\Omega^m_1$ by an element of $G_{r,s}$. 
We obtain this by performing four steps. In each step, we achieve the fulfillment of new conditions on a point from $(\CC^2\setminus\{0\})^m$, while maintaining the conditions obtained in previous steps. Without loss of generality, we assume $r=1$ and $s=2$. 

\smallskip 

{\it Step 1:}\ $x_i\ne 0$ for all $i$. To obtain this, we apply the transformation $(x+\alpha y, y)$ with $\alpha\ne -\frac{x_i}{y_i}$ for all $i$ with $y_i\ne 0$. If $y_i=0$ then $x_i\ne 0$ by our assumption.  

\smallskip 

{\it Step 2:}\ in the collection, there is no pair of the form $(x_i,y_i)$ and $(-x_i,-y_i)$. To get this, we apply the transformation $(x, y+\beta x^2)$ with $\beta\ne -\frac{y_i+y_j}{2x_i^2}$ for all $i, j$ with $x_j=-x_i$. 

\smallskip 

{\it Step 3:}\ $x_i\ne 0$ for all $i$ and $x_i^2\ne x_j^2$ for all $i\ne j$. To achieve this, we apply the transformation $(x+\alpha y, y)$ with 
$$
\alpha\ne -\frac{x_i}{y_i}  \ \ \text{and} \ \ \alpha\ne\frac{x_i-x_j}{y_j-x_i} \ \text{for\, all} \ y_i\ne y_j \ \ \text{and} \ \ \alpha\ne -\frac{x_i+x_j}{y_i+x_j} \ \text{for\, all} \ y_i\ne -y_j.
$$
If  $y_i=y_j$ then $x_i\ne x_j$ and this transformation does not change the inequality. Also if $y_i=-y_j$ then $ x_i\ne -x_j$ and this transformation does not change the inequality as well. 

\smallskip 

{\it Step 4:}\ $y_i\ne 0$ for all $i$ and $y_i\ne y_j$ for all $i\ne j$. To obtain this, we apply the transformation $(x, y+\beta x^2)$ with
$$
\beta\ne -\frac{y_i}{x_i^2}  \ \ \text{and} \ \ \beta\ne\frac{y_i-y_j}{x_j^2-x_i^2}  \ \text{for\, all} \ i\ne j.
$$
\smallskip

The proof of Theorem~\ref{teor2} is completed. 
\end{proof}

Let us give one more application of Proposition~\ref{pimp}. 

\begin{proposition} \label{dan}
Let us consider the Danielewski surface $Y_p := \{ (x,y,z) \in \CC^3 \,:\, xy = p(z)\}$, where $p \in \CC[z]$ has only simple zeroes and is of degree $\nu \geq 2$.
Let
$$
D_1 = p'(z) \frac{\partial}{\partial x} + y \frac{\partial}{\partial z}, \quad
D_2 = p'(z) \frac{\partial}{\partial y} + x \frac{\partial}{\partial z}, \quad
D_3 = y D_1, \quad
D_4 = x D_2.
$$
Then the group generated by flows of the locally nilpotent derivations $D_1, D_2, D_3, D_4$ acts generically infinitely transitively on $Y_p$.
\end{proposition}

\begin{proof}
We apply Proposition \ref{pimp} with $\Theta_1 = D_1$ and $\Theta_2 = D_2$. Here, we have 
\[
\Omega = Y_p \setminus \{p'(z) = 0 \},
\]
see \cite[Lemma 4.12]{An-3}. 
Note that $\Ker(\Theta_1)=\CC[y]$ and $\Ker(\Theta_2)=\CC[x]$. According to \cite[Corollary~4.6]{An-3}, we have
$$
y^{\nu-2} \CC[y] \Theta_1 \oplus x^{\nu-2} \CC[x] \Theta_2 \subseteq \lie(D_1, D_2, D_3, D_4),
$$
and the result follows. 
\end{proof}

\begin{remark}
This action is in fact infinitely transitive, since $d=1$ and one can move any $m$-tuple of points to the subset $\Omega^m_1$
using the flows of $D_1, D_2, D_3, D_4$; see~\cite[Theorem~4.11]{An-3}.
\end{remark}

\smallskip 

Finally, we present one more result. Let $X$ be a smooth affine algebraic variety and $\omega$ be an algebraic volume form on $X$, i.e., a nowhere vanishing section of the canonical bundle. Denote by $L_{\omega}(X)$ the Lie algebra of all polynomial vector fields on $X$ that preserve the form~$\omega$.  Let $D_1,\ldots, D_k$ be locally nilpotent derivations on $\CC[X]$ and $H_i$ be the flow of $D_i$.

\begin{proposition}  \label{vpp} 
If $D_1,\ldots, D_k$ generate the Lie algebra $L_{\omega}(X)$, then the group $G:=\langle H_1,\ldots, H_k\rangle$ acts on $X$ infinitely transitively.  
\end{proposition}

This result follows directly from~\cite[Corollary~8]{An-1}. Its proof is based on the so-called algebraic volume density property and approximation of certain analytic objects by algebraic ones. Proposition~\ref{vpp} and its variations are used in \cite{An-1, An-2, An-3} to prove that some groups generated by finitely many $\GG_a$-subgroups are infinitely transitive. 

\smallskip

The Lie algebra $L$ of all polynomial vector fields on $\CC^2$ with zero divergence coincides with $L_{\omega}(\CC^2)$, where $\omega$ is the standard volume form on $\CC^2$; see Section~\ref{sec2}. So, we obtain 

\begin{corollary}
If locally nilpotent derivations $D_1,\ldots, D_k$  of the algebra $\CC[x,y]$ generate the Lie algebra $L$, then the group $G:=\langle H_1,\ldots, H_k\rangle$ acts on $\CC^2$ infinitely transitively.   
\end{corollary} 

\begin{corollary} \label{two} 
The one-parameter subgroups
$$
(x+\alpha y^2, y) \quad \text{and} \quad (x+\beta y, y-\beta x -\frac{\beta^2}{2}),  \quad \alpha, \beta \in \CC,
$$
generate a group $G$ that is infinitely transitive on $\CC^2$. 
\end{corollary}

\begin{proof}
The two one-parameter subgroups above are the subgroups ${H_1=\exp(\alpha\partial_1)}$ and ${H_2=\exp(\beta\partial_2)}$, where 
$\partial_1= y^2\frac{\partial}{\partial x} $ and $\partial_2= \frac{\partial}{\partial x}-x\frac{\partial}{\partial y}$ are the locally nilpotent derivations from Proposition~\ref{pfirst}.
\end{proof} 

\section*{Acknowledgments}

Some results of this paper were obtained during the stay of the second author at the University of Ljubljana in November 2024, and he would like to thank this institution and especially 
Franc Forstneri\v c for invitation and hospitality. 

\section*{Funding}

The first author was supported by the European Union (ERC Advanced grant HPDR, 101053085 to Franc Forstneri\v{c}). The second author was supported by the grant RSF 25-11-00302.


\end{document}